\documentclass[11pt]{amsart}
\usepackage{lmodern}
\usepackage{amsmath, amsthm, amssymb, amsfonts}
\usepackage[normalem]{ulem}
\usepackage{hyperref}

\usepackage{verbatim} 
\usepackage{longtable}

\usepackage{mathtools}

\usepackage{caption}

\usepackage{tikz-cd}
\usetikzlibrary{arrows}

\theoremstyle{plain}
\newtheorem{thm}{Theorem}[section]
\newtheorem{cor}[thm]{Corollary}
\newtheorem{lem}[thm]{Lemma}
\newtheorem{prop}[thm]{Proposition}
\newtheorem{conj}[thm]{Conjecture}
\newtheorem{question}[thm]{Question}

\theoremstyle{definition}

\theoremstyle{remark}
\newtheorem{rmk}[thm]{Remark}

\newcommand{\BC}{{\mathbb{C}}}
\newcommand{\BD}{{\mathbb{D}}}

\newcommand{\BH}{{\mathbb{H}}}

\newcommand{\BP}{{\mathbb{P}}}
\newcommand{\BQ}{{\mathbb{Q}}}
\newcommand{\BR}{{\mathbb{R}}}

\newcommand{\BZ}{{\mathbb{Z}}}

\newcommand{\CC}{{\mathcal C}}

\newcommand{\CE}{{\mathcal E}}
\newcommand{\CF}{{\mathcal F}}

\newcommand{\CH}{{\mathcal H}}

\newcommand{\CM}{{\mathcal M}}

\newcommand{\CP}{{\mathcal P}}

\DeclareFontFamily{OT1}{rsfs}{}
\DeclareFontShape{OT1}{rsfs}{n}{it}{<-> rsfs10}{}
\DeclareMathAlphabet{\curly}{OT1}{rsfs}{n}{it}

\makeatletter
\let\@wraptoccontribs\wraptoccontribs
\makeatother

\begin{document}
\title[Topology of Lagrangian fibrations and Hodge theory]{Topology of Lagrangian fibrations and Hodge theory of hyper-K\"ahler manifolds}
\date{\today}

\author{Junliang Shen}
\address{Massachusetts Institute of Technology, Department of Mathematics, Simons Building, 77 Massachusetts Avenue, Cambridge, MA 02139, USA}
\email{jlshen@mit.edu}

\author{Qizheng Yin}
\address{Peking University, Beijing International Center for Mathematical Research, Jingchunyuan Courtyard \#78, 5 Yiheyuan Road, Haidian District, Beijing 100871, China}
\email{qizheng@math.pku.edu.cn}

\contrib[with an appendix by]{Claire Voisin}
\address{Coll\`ege de France, 3 rue d'Ulm, 75005 Paris, France}
\email{claire.voisin@imj-prg.fr}

\begin{abstract}
We establish a compact analogue of the $P = W$ conjecture. For a projective irreducible holomorphic symplectic variety with a Lagrangian fibration, we show that the perverse numbers associated with the fibration match perfectly with the Hodge numbers of the total space. This builds a new connection between the topology of Lagrangian fibrations and the Hodge theory of hyper-K\"ahler manifolds. We present two applications of our result, one on the cohomology of the base and fibers of a Lagrangian fibration, the other on the refined Gopakumar--Vafa invariants of a $K3$ surface. Furthermore, we show that the perverse filtration associated with a Lagrangian fibration is multiplicative under cup product.
\end{abstract}

\maketitle

\baselineskip=14.5pt

\setcounter{tocdepth}{1} 

\tableofcontents
\setcounter{section}{-1}

\section{Introduction}

\subsection{Perverse filtrations} Throughout the paper, we work over the complex numbers $\BC$.

Let $\pi: X \rightarrow Y$ be a proper morphism with $X$ a nonsingular algebraic variety. The perverse $t$-structure on the constructible derived category $D_c^b(Y)$ induces an increasing filtration on the cohomology $H^*(X, \BQ)$,
\begin{equation} \label{Perv_Filtration}
    P_0H^\ast(X, \BQ) \subset P_1H^\ast(X, \BQ) \subset \dots \subset P_kH^\ast(X, \BQ) \subset \dots \subset H^\ast(X, \BQ),
\end{equation}
called the \emph{perverse filtration} associated with $\pi$. See Section \ref{sec1} for a brief review of the subject.


The filtration (\ref{Perv_Filtration}) is governed by the topology of the map $\pi: X\rightarrow Y$. Some important invariants are the \emph{perverse numbers}
\[
^\mathfrak{p}h^{i,j}(X)= \dim \mathrm{Gr}^P_i H^{i+j}(X, \BQ)= \dim \left( P_i H^{i+j}(X, \BQ)/ P_{i-1}  H^{i+j}(X, \BQ) \right).
\]

The purpose of this paper is to study the perverse filtrations and perverse numbers of Lagrangian fibrations associated with projective irreducible holomorphic symplectic varieties.


\subsection{Integrable systems} \label{originp=w}
Perverse filtrations emerge in the study of integrable systems. An interesting example is given by the Hitchin system
\begin{equation}\label{Hitchin_fib}
h: \CM_{\mathrm{Dol}} \rightarrow \mathbb{C}^N
\end{equation}
associated with the moduli space $\CM_{\mathrm{Dol}}$ of semistable Higgs bundles on a nonsingular curve $C$. The morphism (\ref{Hitchin_fib}), called the \emph{Hitchin fibration}, is Lagrangian with respect to the canonical holomorphic symplectic form on~$\CM_{\mathrm{Dol}}$ given by the hyper-K\"ahler metric on $\CM_{\mathrm{Dol}}$.

The topology of the Hitchin fibration has been studied intensively over the past few decades. Among other things, a striking phenomenon was discovered by de~Cataldo, Hausel, and Migliorini \cite{dCHM1}. It predicts that the perverse filtration of $\CM_{\mathrm{Dol}}$ matches the weight filtration of the mixed Hodge structure on~the corresponding character variety $\mathcal{M}_B$ via Simpson's nonabelian Hodge theory.

More precisely, Simpson constructed in \cite{Simp} a canonical diffeomorphism between the moduli space $\CM_{\mathrm{Dol}}$ of rank $n$ Higgs bundles and the corresponding character variety $\CM_B$ of rank $n$ local systems.\footnote{Following \cite{dCHM1}, we shall only consider Higgs bundles of degree $1$. The corresponding local systems are on $C \backslash \{\textrm{pt}\}$ with monodromy $e^{\frac{2\pi \sqrt{-1}}{n}}$ around the point. This ensures that no strictly semistable Higgs bundles appear and all local systems are irreducible. In particular, the corresponding moduli spaces $\CM_{\mathrm{Dol}}$ and $\CM_B$ are nonsingular.} The induced canonical isomorphism 
\[
H^\ast(\CM_{\mathrm{Dol}}, \BQ) = H^\ast(\CM_B, \BQ) 
\]
is then \emph{expected} to identify the perverse filtration of $\CM_{\mathrm{Dol}}$ and the weight filtration of $\CM_B$,
\[
P_kH^\ast(\CM_{\mathrm{Dol}}, \BQ) = W_{2k}H^\ast(\CM_{B}, \BQ)= W_{2k+1}H^\ast(\CM_{B}, \BQ), \quad k \geq 0.
\]
Such a phenomenon is now referred to as the \emph{$P=W$ conjecture}. See \cite{dCHM1, dCHM3, SZ, Z} for more details and recent progress in this direction.

By \cite{Shende}, all cohomology classes in $H^*(\CM_B, \BQ)$ are of Hodge--Tate type. In fact, we have
\[
H^*(\CM_B, \BQ) = \bigoplus_{k,d} {^k\mathrm{Hdg}^d(\CM_B)}
\]
with 
\[
^k\mathrm{Hdg}^d(\CM_B) =  W_{2k}H^d(\CM_B, \BQ) \cap F^k H^d(\CM_B, \BC) \cap \overline{F}^k H^d(\CM_B, \BC).
\]
Here $F^*$ is the Hodge filtration and $\overline{F}^*$ is the complex conjugate. Hence, as a consequence of the $P=W$ conjecture, the perverse numbers associated with the Hitchin fibration (\ref{Hitchin_fib}) should equal the Hodge numbers of the corresponding character variety
\[
h^{i,j}(\CM_B) = \mathrm{dim} \left( ^i\mathrm{Hdg}^{i+j}(\CM_B) \right).
\]

\begin{conj}[\cite{dCHM1} Numerical $P=W$] \label{P=W}
We have
\begin{equation*}
^\mathfrak{p}h^{i,j}(\CM_{\mathrm{Dol}}) = h^{i,j}(\CM_B).
\end{equation*}
\end{conj}

The Hodge numbers $h^{i,j}(\CM_B)$ have been studied in \cite{HLR, HRV}. A closed formula was conjectured in \cite[Conjecture 1.2.1]{HLR} in terms of combinatorial data.

Furthermore, the perverse numbers $^\mathfrak{p}h^{i,j}(\CM_{\mathrm{Dol}})$ appear naturally as refined Gopakumar--Vafa invariants, which count curves on the local Calabi--Yau~$3$-fold
\[
T^\ast C \times \BC;
\]
see \cite{CDP, MT} and Section \ref{GVinv} for more discussions.

\subsection{Lagrangian fibrations}
Let $M$ be a projective irreducible holomorphic symplectic variety\footnote{Here $M$ is a simply connected nonsingular projective variety with $H^0(M, \Omega_M^2)$ spanned by a nowhere degenerate holomorphic $2$-form $\sigma$.}, or equivalently, an algebraic compact hyper-K\"ahler manifold. A compact analogue of a completely integrable system is a \emph{Lagrangian fibration}
\begin{equation}\label{lag_fib}
\pi: M \to B    
\end{equation}
with respect to the holomorphic symplectic form $\sigma$ on $M$; see \emph{e.g.}~\cite{Be}.\footnote{By definition, the base of a Lagrangian fibration is always assumed to be normal, and the fibers are always connected.} It can be viewed as a higher-dimensional generalization of a $K3$ surface with an elliptic fibration.

The main result of this paper reveals a perfect match between the perverse numbers $^\mathfrak{p}h^{i,j}(M)$ associated with the Lagrangian fibration (\ref{lag_fib}) and the Hodge numbers $h^{i, j}(M)$. 

\begin{thm}[Numerical ``Perverse = Hodge''] \label{P=H}
For any projective irreducible holomorphic symplectic variety $M$ equipped with a Lagrangian fibration $\pi: M \to B$, we have
\begin{equation}\label{P=H_eq}
^\mathfrak{p}h^{i,j}(M)=h^{i,j}(M).
\end{equation}
\end{thm}

The connection between Theorem \ref{P=H} and Conjecture \ref{P=W} can be seen from the following differential geometric point of view. 

Let $\CM_{\mathrm{dR}}$ denote the de Rham moduli space parametrizing rank $n$ flat bundles. By \cite{Hit, Simp95} (see also \cite[Section 2]{HT} for the twisted case), the complex structure of $\CM_{\mathrm{dR}}$ is obtained via a hyper-K\"ahler rotation from the complex structure of $\CM_{\mathrm{Dol}}$. Moreover, the Riemann--Hilbert correspondence implies that the character variety~$\CM_B$ is biholomorphic to $\CM_{\mathrm{dR}}$,
\[
\CM_{\mathrm{Dol}} \xrightarrow{\textrm{HK rot.}} \CM_{\mathrm{dR}} \xrightarrow{\textrm{bihol.}} \CM_B.  
\]
Hence Conjecture \ref{P=W} suggests that the perverse numbers associated with the Hitchin fibration \eqref{Hitchin_fib} can be calculated from Hodge theory after a certain hyper-K\"ahler rotation of $\CM_{\mathrm{Dol}}$ composed with a biholomorphic transformation. Since the Hodge numbers of a compact hyper-K\"ahler manifold do not change after hyper-K\"ahler rotations or biholomorphic transformations, we can view Theorem \ref{P=H} as a compact analogue of the numerical $P=W$ conjecture.

A stronger and ``more than numerical'' version of Theorem \ref{P=H} is given in Theorem \ref{full_P=W}. Note that unlike the original~$P=W$, here one does \emph{not} expect a naive identification of the perverse and Hodge filtrations. The perverse filtration is topological and stays defined with $\mathbb{Q}$-coefficients after hyper-K\"ahler rotations or biholomorphic transformations, while the Hodge filtration is highly transcendental. Hence the matching necessarily involves nontrivial ``periods''.

Further, we mention some subsequent work concerning perverse filtrations and the $P=W$ phenomenon. We refer to \cite{HLSY} for the relation between the perverse and monodromy weight filtrations associated with (degenerations of) compact hyper-K\"ahler manifolds, and to \cite{dCMS} for the connection between the compact and noncompact/Hitchin worlds.


\subsection{Applications}
We discuss two applications of Theorem \ref{P=H}.

\subsubsection{Cohomology of the base and fibers of a Lagrangian fibration}
A central question in the study of the Lagrangian fibration (\ref{lag_fib}) is the following.

\begin{question}\label{question_lag}
Is the base $B$ always isomorphic to a projective space?
\end{question}

Question \ref{question_lag} has been studied in \cite{Hw, Ma1, Ma2, Ma3}. Notably, Hwang \cite{Hw} gave an affirmative answer assuming $B$ nonsingular. Without this assumption, the answer is known only in dimension up to $4$ by the recent work \cite{BK, HX} based on the earlier \cite{Ou}.

The following theorem answers a cohomological version of Question \ref{question_lag} without any assumption on the base. It also computes the invariant cohomology classes on the fibers of $\pi : M \to B$ in all degrees. This generalizes Matsushita's results \cite{Ma3} and \cite[Lemma 2.2]{Ma4}; see also \cite{Og} which assumes~$B \simeq \BP^n$.

\begin{thm}\label{thm0.3}
Let $M$ be a projective irreducible holomorphic symplectic variety of dimension~$2n$ equipped with a Lagrangian fibration $\pi: M \to B$. Then
\begin{enumerate}
\item[(a)] the intersection cohomology of the base $B$ is given by
\[
\mathrm{IH}^d(B, \BQ) = 
\begin{cases}
\langle \beta^k \rangle, & d=2k;\\
0, & d=2k+1,
\end{cases}
\]
where $\beta$ is an ample divisor class on $B$; in particular, we have 
\[
\mathrm{IH}^d (B ,\BQ) \simeq H^d (\BP^n, \BQ);
\]

\item[(b)] the restriction of $H^d(M, \BQ)$ to any nonsingular fiber $M_b \subset M$ is given by
\[
\mathrm{Im} \left\{H^d(M, \BQ) \to H^d(M_b, \BQ)\right\} = 
\begin{cases}
\langle \eta^k|_{M_b} \rangle, & d=2k;\\
0, & d=2k+1,
\end{cases}
\]
where $\eta$ is a $\pi$-relative ample divisor class on $M$.
\end{enumerate}
\end{thm}

Theorem \ref{thm0.3} is a direct consequence of Theorem \ref{P=H}; see Section \ref{secapps}. Via the bridge established in (\ref{P=H_eq}), the holomorphic symplectic form on~$M$ effectively controls the cohomology of the base and fibers of~$\pi : M \to B$.

In Appendix \ref{appb}, an alternative and more direct proof of Theorem \ref{thm0.3} (b) is given by Claire Voisin. We thank her for kindly agreeing to provide the appendix.

\subsubsection{Curve counting invariants} \label{GVinv}
Perverse numbers play an important role in recent constructions of curve counting invariants. For a Calabi--Yau $3$-fold $X$ and a curve class $\beta \in H_2(X, \BZ)$, Gopakumar and Vafa \cite{GV} predicted integer-valued invariants
\[
n_{g, \beta} \in \BZ
\]
counting genus $g$ curves on $X$ lying in the class $\beta$. The \emph{Gopakumar--Vafa (GV)} or \emph{BPS invariants} are expected to govern all the Gromov--Witten (GW) invariants $N_{g, \beta}$ of~$X$ and $\beta$; see \cite{PT} for more details.

Recently, following the ideas of \cite{HST, KL}, Maulik and Toda \cite{MT} proposed a definition of GV invariants via the Hilbert--Chow morphism
\begin{equation}\label{HC}
\mathrm{hc}: \CM_{\beta} \to \mathrm{Chow}_\beta (X).
\end{equation}
Here $\CM_\beta$ is the moduli space of $1$-dimensional stable sheaves on $X$ with
\[
\mathrm{Supp}(\CF) = \beta, \quad \chi(\CF) = 1,
\] 
and $\mathrm{Chow}_\beta(X)$ is the corresponding Chow variety. When $\CM_\beta$ is nonsingular, the invariants $n_{g,\beta}$ take the form
\[
\sum_{i\in \BZ}\chi \left({^\mathfrak{p}\mathcal{H}}^i(R{\mathrm{hc}}_\ast \mathbb{Q}_{\mathcal{M}_\beta}[\dim\mathcal{M}_\beta])\right)y^i = \sum_{g\geq 0} n_{g, \beta} (y^{\frac{1}{2}}+y^{-\frac{1}{2}} )^{2g}.
\]
In particular, the perverse numbers $^\mathfrak{p}h^{i,j}(\CM_\beta)$ associated with the map (\ref{HC}) determine and refine the GV invariants $n_{g,\beta}$.

If $X$ is the local Calabi--Yau $3$-fold $T^\ast C \times \BC$, the Hilbert--Chow morphism~(\ref{HC}) is induced by the Hitchin system (\ref{Hitchin_fib}); see \cite[Section 9.3]{MT}. Hence Conjecture \ref{P=W} together with \cite[Conjecture 1.2.1]{HLR} provides complete answers to the refined GV invariants of $T^\ast C \times \BC$.

Analogously, if $X$ is the Calabi--Yau $3$-fold $S \times \BC$ with $S$ a $K3$ surface, the map (\ref{HC}) is induced by the Beauville--Mukai system
\[
\pi: M \to \BP^n.
\]
Here $M$ is a projective irreducible holomorphic symplectic variety of $K3^{[n]}$-type with $n = \frac{1}{2}\beta^2 +1$. The following theorem is a direct consequence of Theorem \ref{P=H}. 

\begin{thm}\label{new_thm} Let $X = S \times \BC$ with $S$ a $K3$ surface, and let $\beta \in H_2(X, \BZ)$ be a curve class.\footnote{The result holds equally for certain curve classes $\beta$ on nontrivial $K3$-fibered Calabi--Yau $3$-folds $X$. Here we ask that $\beta$ is supported on a $K3$ surface fiber, and that the base curve of $X$ is transverse to the Noether--Lefschetz locus $\mathrm{NL}_\beta$. In this case, the map \eqref{HC} is also induced by the Beauville--Mukai system; see \cite[Appendix A]{KKP}.} Then
\begin{enumerate}
\item[(a)](Multiple cover formula) the invariant $n_{g,\beta}$ only depends on $g$ and the intersection number $\beta^2$;
\item[(b)](GW/GV correspondence) we have
\begin{equation*}\label{KKV}
\sum_{g\geq 0, \beta>0} N_{g,\beta} \lambda^{2g-2}t^{\beta} = \sum_{g\geq 0, \beta>0, k\geq 1} \frac{n_{g,\beta}}{k} \left( \mathrm{sin}\left( \frac{k\lambda}{2} \right)\right)^{2g-2}t^{k\beta}.
\end{equation*}
Here $N_{g,\beta}$ is the genus $g$ reduced GW invariant of $X$ (or $S$) in the class $\beta$.
\end{enumerate}
\end{thm}


For curve classes $\beta$ with different divisibilities, the fibrations (\ref{HC}) are not deformation equivalent. However, by Theorem \ref{P=H}, the invariants $n_{g,\beta}$ (more generally, the perverse numbers $^\mathfrak{p}h^{i,j}(\CM_\beta)$) only rely on the dimension of the moduli space $\mathcal{M}_\beta$ thus independent of the fibration (\ref{HC}). This verifies Theorem \ref{new_thm} (a).

The reduced GW invariants of a $K3$ surface in all genera and all curve classes were computed in \cite{MPT, PT2}, proving a conjecture of Katz--Klemm--Vafa~\cite{KKV}. Theorem \ref{P=H} provides a direct calculation of all the refined GV invariants of a $K3$ surface in terms of the Hodge numbers of the Hilbert schemes of points on a $K3$ surface, in perfect match with the Katz--Klemm--Vafa formula. This verifies Theorem \ref{new_thm} (b) and offers a conceptual explanation why the enumerative geometry of $K3$ surfaces is related to the Hodge theory of the Hilbert schemes of $K3$ surfaces.

Previously, the GV invariants of a $K3$ surface in primitive curve classes were computed in \cite{HST, KY, KL} using the concrete geometry of the Hilbert schemes of points and moduli spaces of sheaves on a $K3$ surface. See also~\cite{KKP} for another proposal of the refined GV invariants of a $K3$ surface and its relation with the Maulik--Toda proposal.


\subsection{Idea of the proof}
The proof of Theorem \ref{P=H} compares two structures, one on each end of the equality (\ref{P=H_eq}). It manifests some form of symmetry between holomorphic symplectic geometry and hyper-K\"ahler geometry.

On the perverse side, there is an action of $\mathfrak{sl}_2 \times \mathfrak{sl}_2$ on
the graded pieces
\[\bigoplus_i \mathrm{Gr}_i^PH^*(M, \mathbb{Q}),\]
which arises from the decomposition theorem package and keeps track of the perverse grading. An important observation is that for $M$ holomorphic symplectic, this action admits a natural lifting to $H^*(M, \mathbb{Q})$ via the Beauville--Bogomolov form of $M$; see Proposition \ref{prop1.2}.

On the Hodge side, the key ingredient is the action of a special orthogonal Lie algebra on the cohomology $H^*(M, \mathbb{C})$ of a compact hyper-K\"ahler manifold~$M$. An action of~$\mathfrak{so}(5)$ was first discovered by Verbitsky \cite{Ver90}, and was used to keep track of the Hodge grading. It was later generalized in \cite{Ver95, Ver96} to an action of
\[\mathfrak{so}(b_2(M) + 2),\]
where $b_2(M)$ is the second Betti number of~$M$. The same result was also obtained by Looijenga and Lunts \cite[Section 4]{LL}.

Finally, we construct a variation of $\mathfrak{so}(5)$-actions over the Zariski closure of the space of twistor lines. This provides a ``continuous deformation'' from the decomposition theorem package associated with $\pi: M \to B$ to the Hodge decomposition of $M$; see Theorem~\ref{full_P=W}. The dimensions of the weight spaces are conserved under this variation, proving Theorem \ref{P=H}.


A similar argument also yields the multiplicativity of the perverse filtration on $H^*(M, \mathbb{Q})$. This is of independent interest and is included in Appendix \ref{app}. It is worth mentioning that multiplicativity turns out to be the only remaining step in proving the original $P = W$ conjecture following~\cite{dCMS}.

\subsection{Acknowledgements}
We are grateful to Mark Andrea de Cataldo, Javier Fres\'an, Haoshuo Fu, Daniel Huybrechts, Zhiyuan Li, Davesh Maulik, Ben Moonen, Wenhao Ou, Rahul Pandharipande, Claire Voisin, Botong Wang, Chenyang Xu, Zhiwei Yun, and Zili Zhang for helpful discussions about perverse filtrations, the $P=W$ conjecture, and hyper-K\"ahler manifolds. We wish to thank an anonymous referee of our previous submission for a simpler argument for Corollary \ref{conserve}. We also thank all referees for their careful reading and numerous useful suggestions. 

J.~S.~was supported by the NSF grant DMS-2000726. Q.~Y.~was supported by the NSFC grants 11701014, 11831013, and 11890661.

\section{Perverse filtrations} \label{sec1}

We begin with some relevant properties and examples of the perverse filtration associated with a proper surjective morphism $\pi: X\rightarrow Y$. The standard references are \cite{BBD, dCM0, dCM1, WSB}. For simplicity, we assume $X$ nonsingular projective and $Y$ projective throughout this section.

\subsection{Perverse sheaves}
Let $D_c^b(Y)$ denote the bounded derived category of~$\BQ$-constructible sheaves on $Y$, and let $\BD: D_c^b(Y)^{\mathrm{op}} \to D_c^b(Y)$ be the Verdier duality functor. The full subcategories 
\begin{align*}
    ^ \mathfrak{p} D_{\leq 0}^b(Y) & = \left\{\CE \in D_c^b(Y): \dim \mathrm{Supp}(\CH^i(\CE)) \leq -i \right\}, \\
    ^ \mathfrak{p} D_{\geq 0}^b(Y) & = \left\{\CE \in D_c^b(Y): \dim \mathrm{Supp}(\CH^i(\BD\CE)) \leq -i \right\}
\end{align*}
give rise to the \emph{perverse $t$-structure} on $D_c^b(Y)$, whose heart is the abelian category of perverse sheaves,
\[
\mathrm{Perv}(Y) \subset D_c^b(Y).
\]

For $k \in \BZ$, let $^\mathfrak{p}\tau_{\leq k}$ be the truncation functor associated with the perverse~$t$-structure. Given an object $\CC \in D_c^b(Y)$, there is a natural morphism
\begin{equation}\label{trun0}
^\mathfrak{p}\tau_{\le k}\CC \rightarrow \CC.
\end{equation}
For the map $\pi:X \to Y$, we obtain from (\ref{trun0}) the morphism
\[
^\mathfrak{p}\tau_{\le k}R\pi_\ast \BQ_X \rightarrow R\pi_\ast \BQ_X,
\]
which further induces a morphism of (hyper-)cohomology groups
\begin{equation}\label{perv_filt}
H^{d-(\dim X -r)}\Big{(}Y, \, ^\mathfrak{p}\tau_{\le k}(R\pi_\ast \BQ_X[\dim X -r]) \Big{)} \rightarrow H^d(X, \BQ).
\end{equation}
Here
\[r = \dim X \times_Y X - \dim X\]
is the \emph{defect of semismallness}. The $k$-th piece of the perverse filtration (\ref{Perv_Filtration})
\[
P_kH^d(X, \BQ) \subset H^d(X, \BQ)
\]
is defined to be the image of (\ref{perv_filt}).\footnote{Here the shift $[\dim X-r]$ is to ensure that the perverse filtration is concentrated in the degrees $[0, 2r]$.}

\subsection{The decomposition theorem}
The perverse numbers can be expressed through the decomposition theorem \cite{BBD, dCM0}. By applying the decomposition theorem to the map $\pi: X \to Y$, we obtain an isomorphism
\begin{equation}\label{decomp}
R\pi_*\BQ_X[\dim X-r]\simeq \bigoplus_{i=0}^{2r}\mathcal{P}_i[-i] \in D^b_c(Y)
\end{equation}
with $\mathcal{P}_i  \in \mathrm{Perv}(Y)$. The perverse filtration can be identified as
\[
P_kH^d(X,\BQ)=\mathrm{Im}\Big\{H^{d-(\dim X - r)}(Y, \bigoplus_{i=0}^k\mathcal{P}_i[-i])\to H^d(X,\BQ)\Big\},
\]
and hence
\begin{equation*}
\mathrm{Gr}_i^P H^{i+j} (X, \BQ) \simeq  H^{j- (\dim X - r)}(Y, \CP_i).
\end{equation*}
In particular, the perverse numbers are given by the dimensions
\begin{equation} \label{pervform}
^\mathfrak{p}h^{i,j}(X) = h^{j- (\dim X - r)}(Y, \CP_i).
\end{equation}

For simplicity, we assume from now on the condition
\begin{equation}\label{dim_condition}
\dim X = 2 \dim Y = 2r.
\end{equation}
The decomposition (\ref{decomp}) then implies that the perverse number $^\mathfrak{p}h^{i,j}(X)$ is nontrivial only if
\[
0 \leq i\leq 2r,\quad 0 \leq j \leq 2r.
\]
Moreover, there are symmetries
\begin{equation*}
^\mathfrak{p}h^{i,j}(X) =  {^\mathfrak{p}h}^{2r-i,j}(X)=   {^\mathfrak{p}h}^{i,2r-j}(X)
\end{equation*}
obtained from the hard Lefschetz theorems for perverse cohomology groups associated with the map $\pi: X\to Y$.

Hence, under the condition (\ref{dim_condition}), the perverse numbers behave like the Hodge numbers of a compact K\"ahler manifold of dimension~$2r$.\footnote{However, the symmetry ${^\mathfrak{p}h}^{i,j}(X)= {^\mathfrak{p}h}^{j,i}(X)$ does not hold in general even under the condition (\ref{dim_condition}). A counterexample is given by the projection $\pi: C \times \BP^1 \to \BP^1$ with $C$ a higher genus curve. On the other hand, the symmetry $h^{i,j}(X)=h^{2r-i,j}(X)$, which is satisfied by the perverse numbers, does not hold for a general compact K\"ahler manifold of dimension $2r$.}

\subsection{$\BQ[\eta, \beta]$-modules}

Let $\eta \in H^2(X, \BQ)$ be a relative ample class with respect to the map $\pi: X\to Y$, and let $\beta \in H^2(X, \BQ)$ be the pullback of an ample class on the base variety $Y$. The actions of $\eta$ and $\beta$ on the total cohomology $H^\ast(X, \BQ)$ via cup product give morphisms
\begin{align*}
\eta & : P_kH^i(X, \BQ) \rightarrow P_{k+2}H^{i+2}(X, \BQ), \\
\beta & : P_kH^i(X, \BQ) \rightarrow P_kH^{i+2}(X, \BQ).
\end{align*}
By the hard Lefschetz theorems for perverse cohomology groups, they induce an $\mathfrak{sl}_2 \times \mathfrak{sl}_2$-action on the vector space 
\[
\BH = \bigoplus_{i,j \geq 0} \mathrm{Gr}^P_i H^{i+j} (X, \BQ)
\]
which yields a primitive decomposition of $\BH$. Under the the condition (\ref{dim_condition}), the primitive decomposition takes the form
\begin{equation}\label{primitive_decomp}
    \bigoplus_{0 \leq i, j \leq r} \BQ[\eta]/(\eta^{r - i +1}) \otimes \BQ[\beta]/(\beta^{r - j+1}) \otimes P^{i,j} \xrightarrow{\sim} \BH
\end{equation}
with $P^{i, j} \subset \mathrm{Gr}^P_i H^{i+j} (X, \BQ)$; see \cite[Corollary 3.10]{WSB}.

In general, the decomposition (\ref{primitive_decomp}) does not lift canonically to a splitting of the perverse filtration on $H^\ast(X, \BQ)$; see \cite{dC, D}. The following result provides a criterion for the existence of a canonical splitting.

\begin{prop} \label{prop1.2}
Let $\pi: X \to Y$ be a map satisfying (\ref{dim_condition}). Assume further that the actions of $\eta$ and $\beta$ on $H^\ast(X, \BQ)$ via cup product induce an isomorphism of $\BQ[\eta, \beta]$-modules
\begin{equation}\label{V_decomp}
\bigoplus_{0 \leq i, j \leq r} \BQ[\eta]/(\eta^{r-i+1}) \otimes \BQ[\beta]/(\beta^{r-j+1}) \otimes V^{i,j} \xrightarrow{\sim} H^\ast(X, \BQ)
\end{equation}
with $V^{i,j} \subset H^{i+j}(X, \BQ)$. Then we have
\begin{enumerate}
    \item[(a)] $V^{i,j} \subset P_i H^{i+j}(X, \BQ)$;
    \item[(b)] the decomposition (\ref{V_decomp}) splits the perverse filtration, \emph{i.e.},
    \[
    P_kH^d(X, \BQ) = \bigoplus_{\substack{2a+i \leq k \\ 2a+2b+i+j =d}} \eta^a \beta^b \, V^{i,j};
    \]
    \item[(c)] $\dim V^{i,j} = \dim P^{i, j}$, where $P^{i, j}$ is as in (\ref{primitive_decomp}).
\end{enumerate}
\end{prop}

\begin{proof}
We first prove (a). Assume that $v\in V^{i,j}$ does not lie in $P_i H^{i+j}(X, \BQ)$. Then there exists $k >i$ such that 
\[
v\in P_k H^{i+j}(X, \BQ), \quad v\notin P_{k-1} H^{i+j}(X, \BQ).
\]
The hard Lefschetz theorem implies that the equivalence class \[
[v] \in \mathrm{Gr}_k^P H^{i+j}(X, \BQ)
\]
satisfies 
\begin{equation}\label{eqn16}
\beta ^{r- (i+j-k)} [v] \neq 0 \in \mathrm{Gr}_k^P H^{2r - (i + j) + 2k}(X, \BQ). 
\end{equation}
On the other hand, by definition we have $\beta^{r-j+1} v =0$. This contradicts~(\ref{eqn16}) since $r - (i + j - k) \geq  r - j + 1$.

We now prove (b). In \cite[Section 2]{D}, Deligne constructed a splitting of the perverse filtration associated with $\pi: X \to Y$, which only depends on the choice of a relative ample class on $X$. Here we use Deligne's splitting with respect to $\eta$ and we follow closely the notation in \cite[Section 1.4.3]{dCHM1}.

For $d \geq 0$, $0 \leq i \leq r$, and $0\leq j \leq r-i$, Deligne's construction provides a splitting
\begin{equation}\label{Q_decomp}
H^d(X ,\BQ) = \bigoplus_{i,j} Q^{i,j;d}
\end{equation}
which satisfies the following properties:
\begin{enumerate}
    \item[(i)] $P_kH^d(X, \BQ) = \oplus_{i+2j\leq k} Q^{i,j;d}$;
    \item[(ii)] for $0\leq j < r-i$, we have
     \[
     \eta \, Q^{i,j;d} = Q^{i,j+1;d+2};
     \]
     \item[(iii)] \cite[Lemma 1.4.3]{dCHM1} for $v \in P_kH^d(X, \BQ)$, if $\eta^{r-k+1} v=0$, then we have 
     \[
     v\in Q^{k,0;d}.
     \]
\end{enumerate}

We compare our decomposition (\ref{V_decomp}) with Deligne's decomposition (\ref{Q_decomp}). First, by (a) and the property (iii) above, we have
\begin{equation*}
    \bigoplus_{2b + i+j =d} \beta^b \, V^{i,j} \subset  Q^{i,0;d}.
\end{equation*}
Hence for $0\leq a \leq r-i$, we obtain from (ii) that
\begin{equation}\label{inclusion}
\bigoplus_{2a+2b + i+j=d} \eta^a \beta^b \, V^{i,j} \subset Q^{i,a;d}.
\end{equation}
Finally, since
\[
\bigoplus_{i,j,d} \, \bigoplus_{2a+2b + i+j=d} \eta^a \beta^b \, V^{i,j} = H^\ast(X, \BQ) = \bigoplus_{i,a,d}Q^{i,a;d},
\]
every inclusion in (\ref{inclusion}) is an identity. Together with the property (i), this shows (b).

Part (c) is a direct consequence of (b).
\end{proof}

\begin{rmk}
In fact, our argument shows that under the hypothesis of Proposition \ref{prop1.2}, all three splittings in \cite{D} as well as the two additional splittings in \cite{dC} coincide; see \cite[Proposition 2.6.3]{dC}.
\end{rmk}

\subsection{Examples}
We give several concrete examples which demonstrate the necessity of the holomorphic symplectic/Lagrangian condition and the perverse $t$-structure.

We consider $\pi: X \to Y$ to be a proper surjective morphism from a nonsingular surface to a nonsingular curve with connected fibers. In particular, the condition (\ref{dim_condition}) holds. Let $j: Y^\circ \hookrightarrow Y$ be the maximal subcurve of $Y$ such that the restricted map 
\[
\pi^\circ = \pi|_{Y^\circ}: X^\circ \rightarrow Y^\circ
\]
is smooth.

By \cite[Theorem 3.2.3]{dCM123}, the decomposition theorem (\ref{decomp}) for $\pi: X\to Y$ takes the form
\begin{multline}\label{decomp_surface}
    R\pi_\ast \BQ_X[1] \simeq \BQ_Y[1] \oplus \Big( j_\ast R^1\pi^\circ_\ast \BQ_{X^\circ}[1] \oplus \bigoplus_{y\in Y \setminus Y^\circ}\BQ_y^{n_y-1} \Big)[-1] \\  \oplus (\BQ_Y[1])[-2].
\end{multline}
Here $n_y$ denotes the number of irreducible components of the fiber over~$y$, and $\BQ_Y[1]$ and $j_\ast R^1\pi^\circ_\ast \BQ_{X^\circ}[1] \oplus \oplus_{y\in Y \setminus Y^\circ}\BQ_y^{n_y-1}$ are perverse sheaves on $Y$.

\subsubsection{$K3$ surfaces} \label{K3example}
Let $X$ be a $K3$ surface equipped with an elliptic fibration $\pi: X \to Y$. Then we have $Y = \BP^1$, and the formula (\ref{pervform}) together with the decomposition~(\ref{decomp_surface}) implies that 
\[
^\mathfrak{p}h^{0,0}(X) = {^\mathfrak{p}h}^{0,2}(X) = {^\mathfrak{p}h}^{2,0}(X) = {^\mathfrak{p}h}^{2,2}(X) =1.
\]
Since the perverse Leray filtration degenerates at the $E_2$-page, we find
\[
\sum_{i+j=k} {^\mathfrak{p}h}^{i,j}(X) = b_k(X)
\]
with $b_k(X)$ the $k$-th Betti number. Hence
\[
^\mathfrak{p}h^{0,1}(X) = {^\mathfrak{p}h}^{1,0}(X) = {^\mathfrak{p}h}^{1,2}(X) = {^\mathfrak{p}h}^{2,1}(X) =0,\quad ^\mathfrak{p}h^{1,1}(X) =20.
\]
In total, we have
\begin{equation}\label{eqn1}
^\mathfrak{p}h^{i,j}(X) = h^{i,j}(X)
\end{equation}
which verifies Theorem \ref{P=H} for elliptic $K3$ surfaces.

The following examples show that the equality (\ref{eqn1}) breaks down if we replace the $K3$ surface by a rational elliptic surface, or replace the perverse filtration by the naive Leray filtration.

\subsubsection{Elliptic surfaces}
We do a similar calculation for a rational elliptic surface $\pi: X \to \BP^1$. Here $X$ is obtained as the blow-up of $9$ points on $\BP^2$. The nontrivial perverse numbers are 
\[
^\mathfrak{p}h^{0,0}(X) = {^\mathfrak{p}h}^{0,2}(X) = {^\mathfrak{p}h}^{2,0}(X) = {^\mathfrak{p}h}^{2,2}(X) =1, \quad ^\mathfrak{p}h^{1,1}(X) =8.
\]
However, the nontrivial Hodge numbers of $X$ are given by
\[
h^{0,0}(X) = h^{2,2}(X)=1, \quad h^{1,1}(X) = 10.
\]
Hence the equality (\ref{eqn1}) does not hold for an arbitrary elliptic surface.

\subsubsection{Leray versus perverse}

Let $\tau_{\leq k}$ be the truncation functor associated with the standard $t$-structure on $D^b_c(Y)$. Then we define the Leray filtration
\[
L_k H^d(X, \BQ) = \mathrm{Im} \left\{ H^{d}(Y, \, \tau_{\le k}R\pi_\ast \BQ_X ) \rightarrow H^d(X, \BQ)\right\},
\]
and the corresponding numbers 
\[
^\mathfrak{l} h^{i,j} (X) =  \dim \mathrm{Gr}^L_i H^{i+j}(X, \BQ)= \dim \left( L_i H^{i+j}(X, \BQ)/ L_{i-1}  H^{i+j}(X, \BQ) \right).
\]
At first glance, the Leray filtration seems a natural object to encode the topology of the fibration $\pi: X \to Y$. However, the next example shows that the perverse filtration and the perverse numbers behave more nicely in view of Theorem \ref{P=H}.

As before, we consider a $K3$ surface with an elliptic fibration $\pi: X\to \BP^1$. By the decomposition (\ref{decomp_surface}), we find
\[
^\mathfrak{l} h^{2,0} (X) =h^0\Big(Y, \BQ_Y \oplus \bigoplus_{y\in Y\setminus Y^\circ} \BQ_y^{n_y-1}\Big)= 1 + \sum_{y\in Y\setminus Y^\circ} (n_y - 1).
\]
In particular, the numbers $^\mathfrak{l} h^{i,j} (X)$ are sensitive to the topology of the singular fibers of $\pi: X\to \BP^1$, and are not deformation invariant.

On the other hand, the calculation in Section \ref{K3example} shows that the perverse numbers of elliptic $K3$ surfaces are deformation invariant.\footnote{We refer to \cite{dCM} for more discussions about deformation invariance of the perverse numbers.} This matches Theorem \ref{P=H} since the Hodge numbers are deformation invariant.

\section{Hyper-K\"ahler manifolds} \label{sec2}

This section concerns the Hodge side. We first review Verbitsky's construction of an $\mathfrak{so}$-action on the cohomology of a compact hyper-K\"ahler manifold following \cite{LL, Ver90, Ver95, Ver96}. Then we study a certain variation of~$\mathfrak{so}(5)$-actions over the Zariski closure of the space of twistor lines. We shall see in Section~\ref{mainsec} that variations of $\mathfrak{so}(5)$-actions play the same role as hyper-K\"ahler rotations in the conjectural $P=W$ picture. Throughout this section, let $M$ be a compact hyper-K\"ahler manifold of complex dimension $2n$.

\subsection{The $\mathfrak{so}(5)$-action}
A hyper-K\"ahler manifold is a Riemannian manifold~$(M, g)$ which is K\"ahler with respect to three complex structures $I, J$, and $K$ satisfying the quaternion relations
\[
I^2 = J^2 = K^2 = I J K = -1.
\]
We view $(M,I)$ as a complex manifold which admits a hyper-K\"ahler structure, and we use $\omega_I, \omega_J$, and $\omega_K$ to denote the corresponding K\"ahler classes in $H^2(M ,\BC)$.

By the hard Lefschetz theorem, each K\"ahler class $\omega_\ast$ ($\ast \in \{I, J, K\}$) gives rise to a Lefschetz triple $(L_\ast, H, \Lambda_\ast)$ acting on the cohomology $H^\ast(M ,\BC)$,
\[
\rho_\ast: \mathfrak{sl}_2 \rightarrow \mathrm{End} (H^\ast(M ,\BC)).  
\]
Here
\begin{align*}
L_\ast & : H^i(M, \BC) \to H^{i+2}(M, \BC), \quad L_\ast(v) = \omega_* \cup v, \\
H & : H^i(M, \BC) \to H^i(M, \BC), \quad H(v) = (i - 2n)v,
\end{align*}
and the action
\[
\Lambda_\ast: H^{i}(M, \BC) \to H^{i - 2}(M, \BC)
\]
is uniquely determined by $L_\ast$ and $H$. We recall the main result of \cite{Ver90}.

\begin{thm}[Verbitsky \cite{Ver90}] \label{so5}
The three $\mathfrak{sl}_2$-actions $\rho_I, \rho_J$, and $\rho_K$ generate an $\mathfrak{so}(5)$-action on the cohomology $H^\ast(M, \BC)$, with
\[
\mathfrak{h}= \left\langle H, H' = -\sqrt{-1}[L_J, \Lambda_K] \right\rangle \subset \mathfrak{so}(5)
\]
a rank $2$ Cartan subalgebra. Moreover, the corresponding weight decomposition of $H^\ast(M, \BC)$ coincides with the Hodge decomposition of the K\"ahler manifold $(M, I)$.
\end{thm}

For the second statement, let 
\[
H^*(M, \BC) = \bigoplus_{i,j} H^{i,j}(M),
\]
be the Hodge decomposition of $M$. Then for any $v \in H^{i,j}(M)$ we have
\[
H(v) = (i+j-2n)v, \quad H'(v)= (i-j)v.
\]

\subsection{Larger $\mathfrak{so}$-actions}
The cohomology algebra $H^*(M, \BC)$ is a \emph{Frobenius--Lefschetz algebra} in the sense of \cite{LL}. Recall the scaling
\[H : H^i(M, \BC) \to H^i(M, \BC), \quad H(v) = (i - 2n)v.\]
An element $\omega \in H^2(M, \BC)$ is called of Lefschetz type if there is an associated Lefschetz triple $(L_\omega, H, \Lambda_\omega)$ acting on $H^\ast(M ,\BC)$,
\[\rho_\omega: \mathfrak{sl}_2 \rightarrow \mathrm{End} (H^\ast(M ,\BC)).\]
The structure Lie algebra $\mathfrak{g}(M) \subset \mathrm{End}(H^*(M, \BC))$ is defined to be the Lie subalgebra generated by $\rho_\omega$ for all $\omega \in H^2(M, \BC)$ of Lefschetz type.

The following theorem extends the $\mathfrak{so}(5)$-action to a larger $\mathfrak{so}$-action on the cohomology $H^\ast(M ,\BC)$.

\begin{thm}[Looijenga--Lunts {\cite[Section 4]{LL}}, Verbitsky \cite{Ver95, Ver96}] \label{largerso}
The structure Lie algebra $\mathfrak{g}(M)$ is isomorphic to $\mathfrak{so}(b_2(M) + 2)$, where $b_2(M)$ is the second Betti number of $M$.
\end{thm}

\subsection{Twistor lines}
As is explained in \cite[Introduction]{Ver95}, the key to the proof of Theorem \ref{largerso} is the following observation.

Let $q_M(-)$ denote the Beauville--Bogomolov quadratic form on the second cohomology $H^2(M, \BC)$, and let $(-,-)_M$ be the corresponding bilinear form~\cite{B}. Given a hyper-K\"ahler structure $(g, I, J, K)$ on $M$, the three K\"ahler classes 
$\omega_I, \omega_J$, and $\omega_K$ satisfy
\begin{equation} \label{quadrel}
\begin{gathered}
q_M(\omega_I) = q_M(\omega_J) = q_M(\omega_K), \\
(\omega_I, \omega_J)_M = (\omega_J, \omega_K)_M = (\omega_K, \omega_I)_M = 0.
\end{gathered}
\end{equation}

Let $\mathrm{Hyp}$ be the classifying space of the hyper-K\"ahler structures on $M$. We consider the map
\begin{equation} \label{periodhyp}
\mathrm{Hyp} \to H^2(M,\BC)\times H^2(M,\BC) \times H^2(M,\BC)
\end{equation}
which sends a hyper-K\"ahler structure $(g, I, J, K)$ to the triple $(\omega_I, \omega_J, \omega_K)$. By \eqref{quadrel}, the image of \eqref{periodhyp} lies in the quasi-affine variety
\begin{equation} \label{variety_D}
D^\circ \subset H^2(M,\BC)\times H^2(M,\BC) \times H^2(M,\BC)
\end{equation}
defined by
\begin{equation*}
\Big\{(x, y, z) : q_M(x) = q_M(y) = q_M(z) \neq 0, \, (x, y)_M = (y, z)_M = (z, x)_M = 0\Big\}.
\end{equation*}

The key observation is a density statement, which is a consequence of the Calabi--Yau theorem and local Torelli.

\begin{prop} \label{density}
The image of \eqref{periodhyp} is Zariski-dense in~$D^\circ$. In particular, all algebraic relations which hold for triples $(\omega_I, \omega_J, \omega_K)$ still hold for all~$(x, y, z) \in D^\circ$ where the relations are defined.
\end{prop}

\begin{proof}
It was shown in \cite[Theorem 1.5]{Ver95} that the image of (\ref{periodhyp}) is Zariski-dense in the real points $D^\circ(\BR)$ of the complex algebraic variety $D^\circ$. Hence it suffices to verify that $D^\circ(\BR)$ is Zariski-dense in $D^\circ$.

We consider the dominant map
\[
f: D^\circ \rightarrow \mathrm{Gr}(3, b_2(M))  
\]
sending a triple $(x,y,z)\in D^\circ$ to the linear subspace $\langle x, y, z \rangle \subset H^2(M, \BC)$.~If
\[
\langle x, y, z \rangle \in \mathrm{Gr}(3, b_2(M))(\mathbb{R})\]
and the 
restriction of $q_M(-)$ to $\langle x, y, z \rangle$ is positive definite, then the fiber of $f$ over $\langle x, y, z \rangle$ contains a real point and admits a transitive action of $O(3) \times \mathbb{G}_m$ defined over $\mathbb{R}$. In particular, the real points are Zariski-dense in the fiber of~$f$ over a positive definite $3$-space. On the other hand, the positive definite~$3$-spaces are Zariski-dense in~$\mathrm{Gr}(3, b_2(M))(\mathbb{R})$, hence in $\mathrm{Gr}(3, b_2(M))$.

Therefore, the Zariski closure $\overline{D^\circ(\BR)}$ of $D^\circ(\BR)$ contains the fibers of $f$ over a Zariski-dense subset of the base $\mathrm{Gr}(3, b_2(M))$. This shows that~$\overline{D^\circ(\BR)}$ dominates~$\mathrm{Gr}(3, b_2(M))$ with general fibers of the same dimension as the fiber of $f$ over a positive definite $3$-space. Since $D^\circ$ is irreducible, a dimension count yields $\overline{D^\circ(\BR)} = D^\circ$.
\end{proof}

\begin{rmk}
We shall see in the proof of Theorem \ref{full_P=W} that the complex points of $D^\circ$ play an essential role in capturing the perverse filtration associated with a Lagrangian fibration. It is therefore crucial to work over the total complex variety $D^\circ$ rather than the twistor lines lying in $D^\circ(\BR)$.
\end{rmk}

\subsection{Variations of $\mathfrak{so}$-actions} We construct a family of $\mathfrak{so}(5)$-actions on the cohomology $H^*(M, \BC)$.

Let $\omega \in H^2(M, \BC)$. Since the elements of Lefschetz type are Zariski-dense in $H^2(M, \BC)$, the operator
\[L_\omega  : H^i(M, \BC) \to H^{i+2}(M, \BC), \quad L_\omega(v) = \omega \cup v\]
belongs to the structure Lie algebra $\mathfrak{g}(M)$. 

\begin{lem}\label{lem2.5}
An element $\omega \in H^2(M ,\BC)$ with $q_M(\omega) \neq 0$ is of Lefschetz~type.
\end{lem}

\begin{proof}
This is a consequence of Verbitsky's theorem \cite[Theorem 11.1]{Ver95}. More precisely, as a key step in the proof of Theorem \ref{largerso}, Verbitsky constructed a canonical Lie algebra isomorphism between~$\mathfrak{g}(M)$ and the Lie algebra $\mathfrak{g}(H^2(M, \BC))$. Here $\mathfrak{g}(H^2(M, \BC))$ is the structure Lie algebra of a ``virtual'' regular surface whose cohomology ring is given by
\[\langle \mathbf{1}\rangle \oplus H^2(M, \BC) \oplus \langle \Omega\rangle,\]
with $v \cup v' = (v, v')_M\Omega$ for $v, v' \in H^2(M, \BC)$. The isomorphism preserves the element $H$ acting as the scaling operator on the cohomology ring of each side; see \cite[Section 12]{Ver95}. The fact that $\omega$ is of Lefschetz type for~$q_M(\omega) \neq 0$ is verified on the $\mathfrak{g}(H^2(M, \BC))$ side, which is \cite[Lemma~9.1]{Ver95}.
\end{proof}


For $\omega \in H^2(M, \BC)$ with $q_M(\omega) \neq 0$, we denote by $(L_\omega, H, \Lambda_\omega)$ the associated Lefschetz triple and by
\[\rho_\omega: \mathfrak{sl}_2 \rightarrow \mathrm{End} (H^\ast(M ,\BC))\]
the corresponding $\mathfrak{sl}_2$-action.




Then, given a point $(x, y, z) \in D^\circ$, each $* \in \{x, y, z\}$ induces a Lefschetz triple $(L_*, H, \Lambda_*)$ and an $\mathfrak{sl}_2$-action $\rho_*$ on $H^*(M, \BC)$.

\begin{cor} \label{variant}
For every $p = (x, y, z) \in D^\circ$, the three $\mathfrak{sl}_2$-actions $\rho_x, \rho_y$, and $\rho_z$ generate an $\mathfrak{so}(5)$-action
\[\iota_p : \mathfrak{so}(5) \to \mathrm{End}(H^*(M, \BC)).\]
\end{cor}

\begin{proof}
This is a direct consequence of Theorem \ref{so5} and Proposition \ref{density}. In fact, the $\mathfrak{so}(5)$-relations among $L_*$, $H$, and $\Lambda_*$ ($* \in \{I, J, K\}$) have been calculated explicitly in \cite[(2.1)]{Ver90}. Now since $L_*$ and $\Lambda_*$ ($* \in \{x, y, z\}$) exist and are algebraic over $D^\circ$, Proposition \ref{density} implies the same relations among~$L_*$, $H$, and $\Lambda_*$ ($* \in \{x, y, z\}$) for all $(x, y, z) \in D^\circ$.
\end{proof}

A Cartan subalgebra of the $\mathfrak{so}(5)$ in Corollary \ref{variant} is given by
\[
\mathfrak{h}= \left\langle H, H'= -\sqrt{-1}[L_y, \Lambda_z] \right\rangle \subset \mathfrak{so}(5).
\]
There is the associated weight decomposition
\[
H^\ast(M, \BC) = \bigoplus_{i,j} H^{i,j}(\iota_p)
\]
such that for any $v \in H^{i,j}(\iota_p)$, we have
\[
H(v) = (i+j-2n)v, \quad H'(v)= (i-j)v.
\]

By Whitehead's first lemma \cite[III.10, Theorem 13]{J} together with \cite[Theorem A]{NR}, finite-dimensional representations of semisimple Lie algebras are rigid. Since the $\mathfrak{so}(5)$-actions vary algebraically over~$D^\circ$, the dimensions of the weight spaces $H^{i, j}(\iota_p)$ are conserved.

\begin{cor} \label{conserve}
For every pair $p, p' \in D^\circ$, we have
\begin{equation*} \label{conseq}
\dim H^{i,j}(\iota_p) = \dim H^{i,j}(\iota_{p'})
\end{equation*}
for all $i, j$.
\end{cor}

\section{``Perverse \texorpdfstring{$=$}{=} Hodge'' and applications} \label{mainsec}
Combining the structures in Sections \ref{sec1} and \ref{sec2}, we prove a compact analogue of the $P=W$ conjecture; see Theorem \ref{full_P=W}. Theorem \ref{P=H} then follows as a corollary. Finally, we show that Theorem \ref{P=H} implies Theorem~\ref{thm0.3}. From now on, let $M$ be a projective irreducible holomorphic symplectic variety of dimension $2n$ equipped with a Lagrangian fibration $\pi: M \to B$. 

\subsection{``Perverse $=$ Hodge''}

We first view $M$ as a compact hyper-K\"ahler manifold. Recall the subvariety
\[D^\circ \subset H^2(M,\BC)\times H^2(M,\BC) \times H^2(M,\BC)\]
given in (\ref{variety_D}). By Corollary \ref{variant}, there is an $\mathfrak{so}(5)$-action 
\[
\iota_p: \mathfrak{so}(5) \to \mathrm{End}(H^\ast(M, \BC))
\]
attached to every point $p \in D^\circ$.

We now construct an $\mathfrak{sl}_2 \times \mathfrak{sl}_2$-action from $\iota_p$ for any $p\in D^\circ$ as follows. Let $(x,y,z)$ be the coordinates of $p$. We consider the elements
\[
\sigma_p = \frac{1}{2}(y+\sqrt{-1}z) \in H^2(M, \BC),\quad \bar{\sigma}_p= \frac{1}{2}(y- \sqrt{-1} z) \in H^2(M,\BC).
\]
They satisfy
\[
q_M(\sigma_p)=q_M(\bar{\sigma}_p) =0, \quad (\sigma_p, \bar{\sigma}_p)_M \neq 0.
\]
Note that if $p= (\omega_I, \omega_J, \omega_K) \in D^\circ$, then $\sigma_p$ is precisely the class of the holomorphic symplectic form on $M$.

Let $(L_\ast, H, \Lambda_\ast)$ be the Lefschetz triple associated with $\ast \in \{x, y, z\}$.
Then
\[
L_{\sigma_p} = \frac{1}{2}(L_y +\sqrt{-1} L_z), \quad \Lambda_{\sigma_p} =\frac{1}{2}( \Lambda_y - \sqrt{-1} \Lambda_z), \quad H_{\sigma_p} = [L_{\sigma_p}, \Lambda_{\sigma_p}]
\]
form an $\mathfrak{sl}_2$-triple $(L_{\sigma_p}, H_{\sigma_p}, \Lambda_{\sigma_p}) \subset \mathfrak{so}(5)$. Similarly, we define the $\mathfrak{sl}_2$-triple associated with $\bar{\sigma}_p$ by
\[
L_{\bar{\sigma}_p} = \frac{1}{2}(L_y -\sqrt{-1} L_z), \quad \Lambda_{\bar{\sigma}_p} = \frac{1}{2}( \Lambda_y + \sqrt{-1} \Lambda_z), \quad H_{\bar{\sigma}_p} = [L_{\bar{\sigma}_p}, \Lambda_{\bar{\sigma}_p}].
\]

Then $(L_{\sigma_p}, H_{\sigma_p}, \Lambda_{\sigma_p})$ and $(L_{\bar{\sigma}_p}, H_{\bar{\sigma}_p}, \Lambda_{\bar{\sigma}_p})$ induce an $\mathfrak{sl}_2 \times \mathfrak{sl}_2$-action associated with the $\mathfrak{so}(5)$-action $\iota_p$. The connection between the $\mathfrak{sl}_2 \times \mathfrak{sl}_2$-action and the $\mathfrak{so}(5)$-action is the following: a vector $v \in H^\ast(M, \BC)$ has weight
\[
H(v) = (i+j-2n)v, \quad H'(v)=(i-j)v
\]
with respect to the Cartan subalgebra
\[
\mathfrak{h}= \left\langle H, H'= -\sqrt{-1}[L_y, \Lambda_z] \right\rangle \subset \mathfrak{so}(5)
\]
if and only if it has weight
\[
H_{\sigma_p}(v) = (i-n)v, \quad H_{\bar{\sigma}_p}(v) = (j-n)v
\]
under the $\mathfrak{sl}_2 \times \mathfrak{sl}_2$-action. In particular, the weight decompositions of these two actions coincide.

By \cite[Lemma 1.8]{Bo}, if an element $w\in H^2(M, \BC)$ satisfies $q_M(w)=0$, then 
\[
w^{n+1} =0 \in H^{2n+2}(M, \BC).
\]
Hence the weight decomposition of the $\mathfrak{so}(5)$-action $\iota_p$ yields a primitive decomposition 
\begin{equation}\label{primitive_decomp1}
    \bigoplus_{0 \leq i, j \leq n} \BC[\sigma_p]/(\sigma_p^{n-i+1}) \otimes \BC[\bar{\sigma}_p]/(\bar{\sigma}_p^{n-j+1}) \otimes W_p^{i,j} \xrightarrow{\sim} H^\ast(M, \BC)
\end{equation}
for the induced $\mathfrak{sl}_2 \times \mathfrak{sl}_2$-action.

The following theorem connects the representation-theoretic description of $H^\ast(M, \BC)$ to the perverse filtration associated with the Lagrangian fibration~$\pi: M \to B$.

\begin{thm}\label{full_P=W}
Let $M$ be a projective irreducible holomorphic symplectic variety equipped with a Lagrangian fibration $\pi: M \to B$, and let $D^\circ$ be as in (\ref{variety_D}). Then there exist two points~$p, p' \in D^\circ$ such that the weight decomposition of $\iota_p$ splits the perverse filtration associated with $\pi$, and the weight decomposition of $\iota_{p'}$ coincides with the Hodge decomposition of $M$.
\end{thm}

\begin{proof}
We take $p'$ to be the triple $(\omega_I, \omega_J, \omega_K) \in D^{\circ}$ given by the hyper-K\"ahler structure $(g, I,J,K)$ of $M$. To prove the theorem, it suffices to find~$p \in D^{\circ}$ splitting the perverse filtration associated with~$\pi: M \to B$.

Let $\beta \in H^2(M,\BQ)$ be the pullback of an ample class on $B$. For any ample class $\eta'$ on $M$, the signature of the Beauville--Bogomolov form implies that there is a unique $\lambda \in \BQ$ satisfying
\[
q_M(\eta' + \lambda \beta) =0.
\]
We use $\eta$ to denote this $\pi$-relative ample class $\eta' + \lambda \beta \in H^2(M,\BQ)$. In particular, we have
\begin{equation}\label{eqn22}
q_M(\eta) = q_M(\beta) =0, \quad \eta^{n+1} = \beta^{n+1} = 0.
\end{equation}

We claim that there exists a point ${p}=(x,y,z) \in D^\circ$ such that
\[
\sigma_p = \eta, \quad \bar{\sigma}_p =\beta.
\]
In fact, we take
\[
y = \eta + \beta, \quad z= -\sqrt{-1}(\eta-\beta).
\]
Then (\ref{eqn22}) implies that 
\[
q_M(y)=q_M(z) = 2(\eta,\beta) \neq 0, \quad (y,z)_M =\sqrt{-1} \left(q_M(\eta) - q_M(\beta)\right) = 0.
\]
By scaling a nonzero vector of $H^2(M,\BC)$ perpendicular to $y$ and $z$ with respect to~$q_M(-)$, we obtain $x \in H^2(M, \BC)$ with
\[
q_M(x) = q_M(y) = q_M(z) \neq 0, \quad (x,y)_M = (y,z)_M = (z, x)_M = 0,
\]
which gives the point $p=(x,y,z)$ as desired. 

Lastly, we check that $\iota_p$ splits the perverse filtration. From (\ref{primitive_decomp1}), we see that the cohomology $H^\ast(M, \BQ)$ (viewed as a $\BQ[\eta, \beta]$-module) admits a primitive decomposition of the form (\ref{V_decomp}) with $r=n$. Here the $\mathfrak{sl}_2 \times \mathfrak{sl}_2$-action is with $\BQ$-coefficients. Proposition \ref{prop1.2} then implies (after tensoring with $\BC$) that the perverse filtration on $H^\ast(M, \BC)$ associated with~$\pi: M \to B$ has a canonical splitting given by the weight decomposition of the $\mathfrak{so}(5)$-action $\iota_p$. This completes the proof.
\end{proof}

The action 
\[
\iota_p: \mathfrak{so}(5) \rightarrow \mathrm{End}(H^\ast(M, \BC))
\]
attached to a point $p\in D^\circ$ can be viewed as a ``generalized'' hyper-K\"ahler rotation from the original hyper-K\"ahler structure $p'= (\omega_I, \omega_J, \omega_K) \in D^\circ$. By Corollary \ref{conserve}, Theorem~\ref{P=H} is a direct consequence of Theorem \ref{full_P=W}.

In Appendix \ref{app}, we shall see that the weight decomposition of $\iota_p$ is compatible with cup product for all $p \in D^\circ$.


\subsection{Proof of Theorem \ref{thm0.3}} \label{secapps}
We apply the decomposition theorem \cite{BBD} to the Lagrangian fibration $\pi: M \to B$,
\[
R\pi_\ast \BQ_M[n] = \bigoplus_{i = 0}^{2n}  {^\mathfrak{p}\mathcal{H}}^i(R\pi_\ast \BQ_M[n])[-i].
\]

For (a), let $\beta \in H^2(B, \BQ)$ be an ample class. Since $B$ is of dimension~$n$, we have $\beta^n \neq 0$. This gives a nontrivial class
\[
\beta^k  \in \mathrm{IH}^{2k}(B, \BQ)
\]
for every $k \leq n$ by the hard Lefschetz theorem for intersection cohomology. Hence it suffices to show that for $k \leq n$, we have
\[
\dim \mathrm{IH}^d(B, \BQ) \leq
\begin{cases}
1, & d=2k;\\
0, & d=2k+1.
\end{cases}
\]

The morphism $\pi$ is a Lagrangian fibration whose general fibers are nonsingular and connected. Since the perverse sheaf ${^\mathfrak{p}\mathcal{H}}^0(R\pi_\ast \BQ_M[n])$ is semisimple, and its restriction to the open subset $U\subset B$ with nonsingular fibers is the shifted trivial local system $\BQ_U[n]$, we see that 
\[
\mathrm{IC}_B = \mathrm{IC}\left(\BQ_U[n]\right)
\]
is a direct summand component of ${^\mathfrak{p}\mathcal{H}}^0(R\pi_\ast \BQ_M[n])$, \emph{i.e.},
\[
{^\mathfrak{p}\mathcal{H}}^0(R\pi_\ast \BQ_M[n])= \mathrm{IC}_B \oplus \CF, \quad \CF \in \mathrm{Perv}(B).
\]
By the expression \eqref{pervform} of the perverse numbers, we find
\[
\dim \mathrm{IH}^d(B, \BQ) \leq \dim H^{d-n}\left(B, {^\mathfrak{p}\mathcal{H}}^0(R\pi_\ast \BQ_M[n])\right) = {^\mathfrak{p}h^{0,d}(M)}.
\]
On the other hand, the Theorem \ref{P=H} says that for $k \leq n$, we have
\[
{^\mathfrak{p}h^{0,d}(M)} = h^{0,d}(M) = \begin{cases}
1, & d=2k;\\
0, & d=2k+1.
\end{cases}
\]

For (b), let $\eta \in H^2(M, \BQ)$ be a $\pi$-relative ample class. Since $\pi$ is of relative dimension~$n$, we have $\eta^n|_{M_b} \neq 0$ for any nonsingular fiber $M_b \subset M$. This gives a nontrivial class
\[
\eta^k|_{M_b}  \in \mathrm{Im}\left\{H^{2k}(M, \BQ) \to H^{2k}(M_b, \BQ)\right\}
\]
for every $k \leq n$. Hence it suffices to show that for $k \leq n$, we have
\[
\dim \mathrm{Im}\left\{H^d(M, \BQ) \to H^d(M_b, \BQ)\right\} \leq
\begin{cases}
1, & d=2k;\\
0, & d=2k+1.
\end{cases}
\]

Let $B^\circ \subset B$ be the locus supporting nonsingular fibers of $\pi$, and let
\[
\pi^\circ = \pi|_{B^\circ} : M^\circ \to B^\circ
\]
be the restriction. There is a commutative diagram
\[\begin{tikzcd}
H^d(M, \BQ) \arrow{r}{} \arrow{d}{} & H^d(M^\circ, \BQ) \arrow{d}{} \\
H^{-n}\left(B, {^\mathfrak{p}\mathcal{H}}^d(R\pi_\ast \BQ_M[n])\right) \arrow{r}{} & H^0(B^\circ, R^d\pi^\circ_\ast \BQ_{M^\circ}) \arrow{r}{} & H^d(M_b, \BQ),
\end{tikzcd}\]
where the horizontal arrows are given by restrictions, and the vertical arrows are given by edge morphisms of (perverse) Leray spectral sequences. Moreover, the composition
\[
H^d(M, \BQ) \to H^d(M_b, \BQ)
\]
is also the restriction.

Again, by the expression \eqref{pervform}, we find
\begin{multline*}
\dim \mathrm{Im}\left\{H^d(M, \BQ) \to H^d(M_b, \BQ)\right\} \\
\leq \dim H^{-n}\left(B, {^\mathfrak{p}\mathcal{H}}^d(R\pi_\ast \BQ_M[n])\right) = {^\mathfrak{p}h^{d,0}(M)}.
\end{multline*}
This time Theorem \ref{P=H} says that for $k \leq n$, we have
\[
{^\mathfrak{p}h^{d,0}(M)} = h^{d,0}(M) = \begin{cases}
1, & d=2k;\\
0, & d=2k+1.
\end{cases}
\]
The proof of Theorem \ref{thm0.3} is now complete. \qed

\appendix

\section{Multiplicativity of perverse filtrations} \label{app}

This appendix is devoted to the multiplicativity of the perverse filtration associated with a Lagrangian fibration. 
Recall the moduli spaces $\CM_{\mathrm{Dol}}$ and~$\CM_B$ in Section \ref{originp=w}. On one hand, the $P=W$ conjecture implies that the perverse filtration associated with the Hitchin fibration (\ref{Hitchin_fib}) is multiplicative under cup product, \emph{i.e.},
\begin{equation}\label{mult}
    P_kH^d(\CM_{\mathrm{Dol}}, \BQ) \times P_{k'}H^{d'}(\CM_{\mathrm{Dol}}, \BQ) \xrightarrow{\cup} P_{k + k'}H^{d + d'}(\CM_{\mathrm{Dol}}, \BQ).
\end{equation}
On the other hand, the multiplicativity (\ref{mult}) is also expected to play an essential role in the proof of the $P=W$ conjecture. This is because the tautological classes on $\CM_{\mathrm{Dol}}$, whose weights are calculated explicitly on~$\CM_B$ in \cite{Shende}, are ring generators of the cohomology 
\[
H^\ast(\CM_{\mathrm{Dol}}, \BQ) = H^\ast(\CM_B, \BQ).
\]
We refer to \cite{dCHM1, Markman, SZ, Z} for more discussions on tautological classes, multiplicativity, and $P=W$. 

We now prove a compact version of (\ref{mult}). In the spirit of Theorem~\ref{thm0.3}, we shall see that the multiplicativity of the perverse filtration is controlled effectively by Hodge theory.


\begin{thm} \label{multiply}
For a projective irreducible holomorphic symplectic variety~$M$ equipped with a Lagrangian fibration $\pi : M \to B$, the perverse filtration on $H^*(M, \BQ)$ is multiplicative under cup product, \emph{i.e.},
\[P_kH^d(M, \BQ) \times P_{k'}H^{d'}(M, \BQ) \xrightarrow{\cup} P_{k + k'}H^{d + d'}(M, \BQ).\]
\end{thm}

Recall the variety $D^\circ$ introduced in \eqref{variety_D}. By Corollary \ref{variant}, each point $p = (x, y, z) \in D^\circ$ gives rise to an $\mathfrak{so}(5)$-action
\[\iota_p : \mathfrak{so}(5) \to \mathrm{End}(H^*(M, \BC)).\]
There is a weight decomposition
\[
H^\ast(M, \BC) = \bigoplus_{i,j} H^{i,j}(\iota_p)
\]
with respect to the Cartan subalgebra
\[
\mathfrak{h}= \left\langle H, H'= -\sqrt{-1}[L_y, \Lambda_z] \right\rangle \subset \mathfrak{so}(5).
\]
Here for any $v \in H^{i,j}(\iota_p)$, we have
\[
H(v) = (i+j-2n)v, \quad H'(v)= (i-j)v.
\]

In view of Theorem \ref{full_P=W} and the canonical splitting of the perverse filtration on $H^*(M, \BC)$, Theorem \ref{multiply} is implied by the following more general statement.

\begin{prop} \label{multiplyy}
For every $p \in D^\circ$, we have
\begin{equation} \label{multeq}
H^{i,j}(\iota_p) \times H^{i',j'}(\iota_p) \xrightarrow{\cup} H^{i + i', j + j'}(\iota_p).
\end{equation}
\end{prop}

\begin{proof}
We first show that Poincar\'e duality satisfies
\begin{equation} \label{pd}
H^{i,j}(\iota_p)^* = H^{2n - i, 2n - j}(\iota_p).
\end{equation}
For this we look at the class of the diagonal
\[\Delta \in H^{4n}(M \times M, \BC).\]
The duality $H^d(M, \BC)^* = H^{4n - d}(M, \BC)$ corresponds to the K\"unneth component
\[\Delta^{4n - d, d} \in H^{4n - d}(M, \BC) \otimes H^d(M, \BC).\]

We consider the induced $\mathfrak{so}(5)$-action on the tensor product
\[H^*(M \times M, \BC) = H^*(M, \BC) \otimes H^*(M, \BC).\]
Here for any $v \in H^{i, j}(\iota_p) \otimes H^{i', j'}(\iota_p)$, we have
\[H(v) = (i+i'+j+j'-4n)v, \quad H'(v)= (i+i'-j-j')v.\]
The duality \eqref{pd} is equivalent to
\[\Delta^{4n - d, d} \in \bigoplus_{i + j = d} H^{2n - i, 2n - j}(\iota_p) \otimes H^{i, j}(\iota_p).\]
Then it is enough to show that $H'(\Delta) = 0$.

If $p = (\omega_I, \omega_J, \omega_K) \in D^\circ$ for some hyper-K\"ahler structure $(g, I, J, K)$, then \eqref{pd} holds by Hodge theory, and hence $H'(\Delta) = 0$. Now since $H'$ varies algebraically over $D^\circ$, the density statement of Proposition \ref{density} implies~$H'(\Delta) = 0$ for all $p \in D^\circ$.

We use exactly the same argument to prove \eqref{multeq}. The cup product on~$H^*(M, \BC)$ is governed by the class of the small diagonal
\[\Delta_3 \in H^{8n}(M \times M \times M, \BC).\]
More precisely, the product 
\[H^d(M, \BC) \times H^{d'}(M, \BC) \xrightarrow{\cup} H^{d + d'}(M, \BC)\]
corresponds via Poincar\'e duality to the K\"unneth component
\[\Delta_3^{4n - d, 4n - d', d + d'} \in H^{4n - d}(M, \BC) \otimes H^{4n - d'}(M, \BC) \otimes H^{d + d'}(M, \BC).\]

This time we consider the induced $\mathfrak{so}(5)$-action on the tensor product
\[H^*(M \times M \times M, \BC) = H^*(M, \BC) \otimes H^*(M, \BC) \otimes H^*(M, \BC).\]
Then for any $v \in H^{i,j}(\iota_p) \otimes H^{i',j'}(\iota_p) \otimes H^{i'',j''}(\iota_p)$, we have
\[H(v) = (i+i'+i''+j+j'+j''-6n)v, \quad H'(v)= (i+i'+i''-j-j'-j'')v.\]
Using \eqref{pd}, we see that \eqref{multeq} is equivalent to
\[\Delta_3^{4n - d, 4n - d', d + d'} \in \bigoplus_{\substack{i + j = d \\ i'+j' = d'}} H^{2n - i, 2n - j}(\iota_p) \otimes H^{2n - i', 2n - j'}(\iota_p) \otimes H^{i + i', j + j'}(\iota_p).\]
Again, it is enough to show that $H'(\Delta_3) = 0$. For $p = (\omega_I, \omega_J, \omega_K) \in D^\circ$, this holds by Hodge theory, and the general case follows from Proposition~\ref{density}. This completes the proof of Proposition \ref{multiplyy}.
\end{proof}

\section{A note on the topology of fibrations\texorpdfstring{\\}{} by Claire Voisin} \label{appb}

We give in this appendix an alternative proof of Theorem \ref{thm0.3} (b), that is,  the following statement which in degree $k=2$ had been proved by \mbox{Matsushita} (see \cite[Lemma 2.2]{Ma4}).

\begin{thm}\label{theoappli}
Let $X$ be compact hyper-K\"{a}hler manifold of dimension $2n$, and $f:X\rightarrow B$ a Lagrangian fibration with general fiber $F$. Then the restriction map $H^k(X,\mathbb{Q})\rightarrow H^k(F,\mathbb{Q})$ is $0$ for $k$ odd and has rank $1$ for $k$ even, $k\leq 2n$.
\end{thm}

The proof given  below  is  in fact possibly equivalent to the one given by the authors, but it is written in a different language. Our presentation  avoids any representation theory and rests on the following  Proposition \ref{theomain} that seems to be of independent interest. Let $X$ be a smooth connected  projective or compact K\"ahler manifold and $\phi:X\rightarrow B$ be a surjective holomorphic map, where $B$ is an analytic space. Let $F\subset X$ be the general fiber
of $\phi$.  It is thus a smooth submanifold of $X$, that we can assume to be connected after Stein factorization. We denote by $[F]$ its cohomology class.
\begin{prop} \label{theomain} \leavevmode
\begin{enumerate}
\item[(i)] The kernel of the map
${j}^*:H^*(X,\mathbb{Q})\rightarrow H^*({F},\mathbb{Q})$ is equal to the kernel
of the cup-product map
$[F]\cup{}: H^*(X,\mathbb{Q})\rightarrow H^{*+2r}(X,\mathbb{Q})$.

\item[(ii)] Equivalently, the intersection pairing $\langle\,,\,\rangle_{{F}}$ is nondegenerate on ${\rm Im}\,{j}^*$.
\end{enumerate}
\end{prop}

\begin{rmk}
It is proved in \cite{Voi92} that for any closed  complex subvariety $F$  of $X$,  the statement above holds for the restriction map
${j}^*:H^k(X,\mathbb{Q})\rightarrow H^k({F},\mathbb{Q})$ when $k\leq2$. However this is not true without assumptions on  $F$  starting from degree $\geq3$. For example,
start from an elliptic curve  $D\subset \mathbb{P}^3$, and let $X$ be the blow-up of $ \mathbb{P}^3$ along $D$, with exceptional divisor $E\subset X$.
Then the restriction map $H^3(X,\mathbb{Q})\rightarrow H^3(E,\mathbb{Q})$ is nonzero, while the cup-product map
$[E]\cup{}:H^3(X,\mathbb{Q})\rightarrow H^5(X,\mathbb{Q})=0$ is zero.
\end{rmk}

\begin{proof}[Proof of Proposition \ref{theomain}]  The equivalence of the two statements is proved as  follows. If $\alpha\in H^k(X,\mathbb{R}) $ satisfies $\alpha_{\mid F}\not=0$, then (ii) says that there is a class $\beta\in H^{2d-k}(X,\mathbb{R})$, where  $d={\rm dim}\,F$, such that $\langle \alpha_{\mid F},\beta_{\mid F}\rangle_F\not=0$. As this is equal to $\langle [F]\cup \alpha,\beta\rangle_X$ we conclude that $ [F]\cup \alpha\not=0$, so (ii) implies~(i).
Conversely, let   $\alpha\in H^*(X,\mathbb{R})$ such that
$\alpha_{\mid F}\not=0$. Then (i) implies that $ [F]\cup \alpha\not=0$ in $H^{k+2r}(X,\mathbb{R})$, where $r={\rm codim}\,F$, and  by Poincar\'{e} duality on~$X$, there exists  a class  $\beta\in  H^{2d-k}(X,\mathbb{R})$ such that
$\langle [F]\cup \alpha,\beta\rangle_X\not=0$, that is  $\langle \alpha_{\mid F},\beta_{\mid F}\rangle_F\not=0$. Thus the intersection pairing $\langle\,,\,\rangle_{{F}}$ is nondegenerate on~${\rm Im}\,{j}^*$ and (i) implies (ii).

We now prove (ii). Let $h\in H^2(X,\mathbb{R})$ be a K\"ahler class and let $h_F$ be its restriction to $F$.
 The Lefschetz decomposition of $H^*(F,\mathbb{R})$ relative to $h_F$ provides for any
$\alpha_F\in H^k(F,\mathbb{R})$ a unique decomposition
\begin{equation*}\label{eqdecomp}
\alpha_F=\bigoplus _{0\leq k-2l\leq d} h_F^l\,\alpha_{F,l},
\end{equation*}
where $\alpha_{F,l}\in H^{k-2l}(F,\mathbb{R})_{\rm prim}$  is primitive with respect to $h_F$, that is, annihilated by $h_F^{d-k+2l+1}\cup{}$.
We have

\begin{lem}\label{lele} The image $I^*:={\rm Im}\,(j^*: H^*(X,\mathbb{R})\rightarrow H^*(F,\mathbb{R}))$ is stable under the Lefschetz decomposition of $H^*(F,\mathbb{R})$ relative to $h_F$.
\end{lem}

\begin{proof}
Let $U\subset B$ be the Zariski (or Zariski analytic) open subset over which  $\phi_U: X_U:=\phi^{-1}(U)\rightarrow U$ is smooth. Then
the relative Lefschetz  decomposition associated to $h$  decomposes the local system $R\phi_{U*}\mathbb{R}$. If
$\alpha\in H^k(X,\mathbb{R})$, the restriction of $ \alpha$ to the fibers of $\phi_U$   provides a constant global section
$\tilde{\alpha}$  of  $R^k\phi_{U*}\mathbb{R}$. The components $\tilde{\alpha}_l$ of the relative
Lefschetz decomposition of~$\tilde{\alpha}$ are also global constant sections of
$R^{k-2l}\phi_{U*}\mathbb{R}$ on $U$. By Deligne's global invariant cycle theorem \cite{D0}, there exist
classes $\alpha_l\in H^{k-2l}(X,\mathbb{R})$ inducing the sections  $\tilde{\alpha}_l$, which means that
${\alpha_{F,l}}={\alpha_l}_{\mid F}$. This proves the lemma.
\end{proof}

Lemma \ref{lele} implies now Proposition \ref{theomain} (ii) using the  Hodge--Riemann bilinear relations.
The Lefschetz decomposition on $I^*$  is a decomposition   into Hodge substructures, and  it
satisfies   the Lefschetz isomorphism property saying that
$h_F^{d-k}\cup{}: I^k\rightarrow I^{2d-k}$
is an isomorphism.
In order to prove (ii), it thus suffices to prove that
the Lefschetz intersection pairing
$$ \langle \alpha,\beta\rangle_{F,h}=\langle h_F^{d-k}\cup\alpha\cup \beta\rangle_F$$
is nondegenerate on $I^k$ for $k\leq d$.
However we know that each $ I^k$ is a real Hodge structure  stable under
the Lefschetz decomposition
$$I^k=\bigoplus_{2l\leq k} h_F^l \, I^{k-2l}_{{\rm prim}}$$
which is orthogonal for $ \langle \,,\,\rangle_{F,h}$, so each term $ I^{k-2l}_{{\rm prim}}$ is a real Hodge substructure
of $H^{k-2l}(F,\mathbb{R})_{\rm prim}$ with its induced Lefschetz intersection pairing. It is finally well-known that the restriction of  $ \langle \,,\,\rangle_{F,h}$ on any real   Hodge substructure of $H^{k-2l}(F,\mathbb{R})_{\rm prim}$ is nondegenerate
because   $ \langle \,,\,\rangle_{F,h}$ polarizes it, so the corresponding Hermitian pairing has a definite sign on each $(p,q)$-component.
\end{proof}

We now turn to the case of a  Lagrangian fibration $X \rightarrow B$ on a compact  hyper-K\"ahler $2n$-fold. The class of the fiber $F$ is proportional to $l^n$, where $l=c_1(L)\in H^2(X,\mathbb{Z})$ and $L$  generates the Picard group of the base. Using the Beauville--Fujiki relations, the class $l$ satisfies $q(l)=0$, where $q$ is the Beauville--Bogomolov form, because $l^{2n}=0$.

\begin{proof}[Proof of Theorem \ref{theoappli}] By Proposition \ref{theomain}, we have to show that the map
\begin{equation} \label{eqcupmap}
l^n\cup{}:H^k(X,\mathbb{C})\rightarrow H^{k+2n}(X,\mathbb{C})
\end{equation}
is $0$ when $k$ is odd and has rank  $1$ when $k\leq 2n$ is even. Note that this rank is at least $1$ in the second case, because the $s$-th powers of a K\"{a}hler class do not vanish on the fiber for $s\leq n$. So we just have to prove that the ranks of the maps (\ref{eqcupmap}) take at most these values.
Let $Q\subset \mathbb{P}(H^2(X,\mathbb{C}))$ be the quadric defined by $q$. It is well-known that
an (Euclidean) open set of $Q$ consists of classes $\sigma_t$ of holomorphic forms on a deformation $X_t$ of $X$.
By semicontinuity of the rank, it suffices to show that for
$\sigma_t\in H^{2,0}(X_t)$, the map
$$ \sigma_t^n \cup{}:H^k(X,\mathbb{C})\rightarrow H^{k+2n}(X,\mathbb{C})$$
is $0$ when $k$ is odd and has rank  $\leq 1$ when $k\leq 2n$ is even. This is however obvious because
$\sigma_t^n$ is of type $(2n,0)$ on $X_t$, hence the image of $ \sigma_t^n \cup{}$ is contained
in $H^{2n,*}(X_t)$, which is $0$ when $*$ is odd and of rank $1$ when $*$ is even and not greater than $2n$.
\end{proof}

\end{document}